\documentclass{amsart}

\usepackage{amssymb,amsfonts,latexsym}

\usepackage{amsmath}
\usepackage{url}

\theoremstyle{plain}

\numberwithin{equation}{section}

\newtheorem{theorem}{Theorem}[section]

\newtheorem{lemma}[theorem]{Lemma}

\newtheorem{proposition}[theorem]{Proposition}

\newtheorem{corollary}[theorem]{Corollary}

\theoremstyle{definition}

\newtheorem *{Lmt}{Lmt}
\newtheorem *{Theorem A}{Theorem A}
\newtheorem *{Theorem B}{Theorem B}
\newtheorem *{Theorem C}{Theorem C}
\newtheorem *{Theorem D}{Theorem D}
\newtheorem *{Corollary B}{Corollary B}
\newtheorem *{Problem A}{Problem A}
\newtheorem *{Problem B}{Problem B}
\newtheorem *{Problem 1}{Problem 1}
\newtheorem *{Problem 2}{Problem 2}

\newtheorem{remark}[theorem]{Remark}

\newcommand{\Ext}{\text{\rm Ext}}

\newcommand{\ach}{\check{A}}

\newcommand{\ben}{\begin{enumerate}}

\newcommand{\een}{\end{enumerate}}

\newcommand{\C}{{\mathbb C}}

\begin{document}

\title
[]
{On groups of $I$-type and involutive Yang-Baxter groups}

\author{Nir Ben David}

\address{Department of Mathematics, Technion-Israel Institute of Technology, Haifa 32000, Israel}
\email{benda@tx.technion.ac.il}

\author{Yuval Ginosar}

\address{Department of Mathematics, University of Haifa, Haifa 3498838, Israel}
\email{ginosar@math.haifa.ac.il}

\date{\today}


\begin{abstract}
We suggest a cohomological framework to describe groups of
$I$-type and involutive Yang-Baxter groups. These groups are key in the study of involutive
non-degenerate set-theoretic solutions of the quantum Yang-Baxter equation.
Our main tool is a lifting criterion for 1-cocycles, established here in a general non-abelian setting.
\end{abstract}

\maketitle

\section{Introduction}
Two families of solvable groups concern us herein.
Groups of $I$-type (or {\it structure groups}) 
were introduced in \cite{ess,GIVdB98} in order to study set-theoretic solutions of the celebrated quantum
Yang-Baxter equation \cite{Y}.
A group is of $I$-type if it carries an {$I$-datum},
i.e. a bijective 1-cocycle whose values lie
in a free abelian group endowed with a permutation action (see the precise definitions in \S\ref{dr}).
A group may admit various $I$-data, and consequently may be of $I$-type in more than one way.
A group of $I$-type has an associated finite quotient which carries an associated
$I$-datum. Such quotients, namely
involutive Yang-Baxter (IYB) groups, are
exactly the adjoint groups of {\it braces} \cite{R05,R07}.
A consequence of the above bijectivity property is that groups of $I$-type, as well as IYB groups are solvable
\cite[Theorem 2.15]{ess}.

The reader is referred to \cite{CJdR} for a thorough survey of the one-to-one correspondence between
involutive non-degenerate set-theoretic solutions of the quantum Yang-Baxter equation and groups of $I$-type.
More details can be found in \cite{bc,CJO,CJO14,GIVdB06,JO05,JO07,LYZ}.

Two problems were posed in \cite{CJdR} in attempt to characterize the family of
groups of $I$-type and by that to describe all involutive
non-degenerate set-theoretic solutions of the quantum Yang-Baxter
equation:
\begin{Problem A}
Classify the IYB groups. In particular, is every
finite solvable group an IYB group?\\
\end{Problem A}
\begin{Problem B} Describe all $I$-data of groups $G$ of $I$-type which
``lie above" a given $I$-datum of an IYB group $G_0$.
\end{Problem B}

Leaning on an idea of W. Rump \cite[\S 12]{R13}, D. Bachiller has recently disproved
the conjecture in Problem A, by presenting a finite nilpotent group which is not IYB \cite{B}.
The classification problem is still challenging.

Also recently, D. Bachiller and F. Ced\'{o} have solved important cases of Problem B
applying braces techniques \cite{bc}.

This note suggests a cohomological approach to tackle both
problems. Lemma \ref{lift} gives a criterion for lifting 1-cocycles
from a quotient of a group to the group itself.
Using the correspondence in this lemma, Theorem \ref{B} describes all groups $G$ of $I$-type
with $I$-data that lie above a given $I$-datum on their associated IYB group $G_0$.
This description is given in terms of $G_0$-module extensions.

As for Problem A, the subfamily of IYB groups established in \cite{CJdR} contains, not merely however,
finite nilpotents of
class 2, abelian-by-cyclic groups and cyclic-by-two generated abelian $p$-groups. Furthermore, it is shown
that any finite solvable group can be embedded in an IYB group, and that
the family of IYB groups is closed to Hall subgroups, to direct products and to wreath products.
Our method can retrieve some of the above families as explained in \S\ref{IYBfam}.

\section{Definitions}\label{dr}
We adopt the definition of groups of $I$-type given in \cite{ess}.
In order to compute the corresponding set-theoretic solutions, it is more convenient to work with their
group presentation \cite[\S 1]{CJdR}.
Let $\mathbb{Z}^n$ be a free abelian group of rank $n$ endowed with the natural action of the symmetric group $S_n$
on a given set of generators.
Then by the definition of the corresponding semidirect product $\mathbb{Z}^n\rtimes S_n$, the natural projection
$$ \begin{array} {ccl}
\mathbb{Z}^n\rtimes S_n & \to & \mathbb{Z}^n\\
(t,\sigma) & \mapsto & t,
\end{array}t\in\mathbb{Z}^n, \sigma\in S_n$$
satisfies the 1-cocycle condition, where $\mathbb{Z}^n$ is a $\mathbb{Z}^n\rtimes S_n$-module via the quotient $S_n$.
A subgroup $G<\mathbb{Z}^n\rtimes S_n$ is of {\it I-type} if the restriction
\begin{eqnarray} \label{projcoc}
\pi:G\to \mathbb{Z}^n
\end{eqnarray} of the above 1-cocycle to $G$ is bijective.
In other words,
$$G=\{(a,\Phi(a))| a\in \mathbb{Z}^n\}$$
for some map $\Phi:\mathbb{Z}^n\rightarrow S_n$.
We call the triple $(G,\mathbb{Z}^n,\pi)$ an {\it ($n$-fold) $I$-datum}
\footnote{this datum, together with the $G$-module structure on $\mathbb{Z}^n$,
is denoted a {\it bijective cocycle quadruple} in \cite{ess}} on the group $G$.
It turns out that a 1-cocycle is bijective if and
only if so are all the 1-cocycles in its cohomology class
(\cite[\S 1.1]{b}, see also \cite[Proposition 4.1]{bg}).
The fact that bijectivity is a class property is respected
by the cohomological structures in \S\ref{l1coc} and \S\ref{exploit}.

Fix an $I$-datum $(G,\mathbb{Z}^n,\pi)$. Let $K$ be the kernel of the action of a group $G$ of $I$-type on $\mathbb{Z}^n$.
Then certainly $K$ is of finite index in $G$, and the restriction of $\pi$ to $K$ is a group-isomorphism.
Consequently, the finite group
$$G_0:=G/K (\hookrightarrow S_n)$$
acts on
$$A:=\mathbb{Z}^n/\pi(K),$$
and the 1-cocycle $\pi$ determines a 1-cocycle
 \begin{eqnarray} \label{defpi0}
 \begin{array} {rcl}
\pi_0:G_0 & \to & A\\
gK & \mapsto & \pi(g)+\pi(K)
\end{array}
\end{eqnarray}
which is bijective as well.
The finite group $G_0$ is termed {\it involutive Yang-Baxter},
and the triple $(G_0,A,\pi_0)$ is the {\it associated $I$-datum} with respect to the given $I$-datum $(G,\mathbb{Z}^n,\pi)$.
It has already been noticed \cite[Theorem 2.1]{CJdR} that
a bijective 1-cocycle $\pi_0\in Z^1(G_0,A)$ from any finite group $G_0$ to
a $G_0$-module $A$ (of the same cardinality) is always associated to
some $I$-datum $(G,\mathbb{Z}^n,\pi)$. Then an I-datum is also sufficient for a finite group to be IYB.
Note that other choice of a 1-cocycle cohomologous to $\pi$ in \eqref{projcoc} yields, in turn,
a 1-cocycle cohomologous to $\pi_0$ in \eqref{defpi0}.

We remark that $I$-data $(G_0,A,\pi_0)$ on a finite group $G_0$ were used to construct
{\it non-degenerate} classes in $H^2(\mathcal{G},\C^*)$
for the semi-direct product $\mathcal{G}=\ach\rtimes G_0$ \cite{eg2,eg3} or,
more generally, for any extension
$$1\rightarrow \ach\rightarrow\mathcal{G}\rightarrow G_0\rightarrow 1:[\beta]\in H^2(G_0,\ach)$$
such that $[\beta]\cup [\pi_0]=0\in H^3(G_0,\C^*)$ \cite{bg}.

\section{Lifting 1-cocycles}\label{l1coc}
The main endeavor throughout this paper is a construction of cohomology classes on groups
that lift given classes on their quotients. To do so in a general non-abelian setting,
we implement the terminology of
\cite[Chapter VII, Appendix]{s}.

Let
\begin{equation}\label{extG}
1\to G_1\to G\xrightarrow{} G_0\to 1
\end{equation}
be an extension of groups, and let
\begin{equation}\label{extgamma}
1\to \Gamma_1\xrightarrow{\iota} \Gamma\xrightarrow{} \Gamma_0\to 1
\end{equation}
be an extension of (non-abelian) $G$-groups via the quotient $G_0=G/G_1$.

Under this general setup, 1-cocycles of $G$ and $G_0$ over the non-abelian modules $\Gamma,\Gamma_1,\Gamma_0$
can still be defined. We shall also work with the well-defined pointed set $H^1(G_0,\Gamma_0)$,
which is identified with the well known cohomology group in case $\Gamma_0$ is abelian \cite[page 123]{s}.

With the above notation, let $\pi:G\to \Gamma$ be a generalized 1-cocycle in $Z^1(G,\Gamma)$ such that
$\pi(G_1)\subset \Gamma_1$.
The corresponding restriction $\pi_1:G_1\to \Gamma_1$ is a group-homomorphism (since the $G_1$-action is trivial).
Next, $\pi$ determines a well defined map
$$ \begin{array} {rcl}
\pi_0:G_0 & \to & \Gamma_0\\
gG_1 & \mapsto & \pi(g)\Gamma_1,
\end{array}$$
which is a generalized 1-cocycle in $Z^1(G_0,\Gamma_0)$ as can easily be shown.
We say that the 1-cocycle $\pi$ {\it lifts} the pair $(\pi_1,\pi_0)$.

We focus on the special case where $\Gamma_1$ is central in $\Gamma$.
Under this assumption, it is not hard to verify that $\pi_1$ is a $G$-invariant morphism,
that is for every $g\in G$ and $n\in G_1$
$$\pi_1(n)=g(\pi_1(g^{-1}ng)). $$
It turns out that the invariant morphism $\pi_1\in$Hom$(G_1,\Gamma_1)^G$ and the generalized 1-cocycle
$\pi_0\in Z^1(G_0,\Gamma_0)$
(or, more precisely, its class) share a common image under two distinct cohomological maps as follows.
Let
$$\text{ Tra : Hom}(G_1,\Gamma_1)^G\to H^2(G_0,\Gamma_1)$$
be the classical transgression map (see \eqref{deftra} herein), and let
$$\Delta:H^1(G_0,\Gamma_0)\to H^2(G_0,\Gamma_1)$$
be the coboundary map (of pointed sets, see \eqref{defcob} herein).
We have the following necessary and sufficient lifting criterion.
\begin{lemma}\label{lift}
Let \eqref{extG} be an exact sequence of groups and let \eqref{extgamma} be a central exact sequence of
(non-abelian) $G$-modules via its quotient $G_0$.
Let $\pi_1\in$ Hom$(G_1,\Gamma_1)^G$ and $\pi_0\in Z^1(G_0,\Gamma_0)$.
Then there exists a 1-cocycle $\pi\in Z^1(G,\Gamma)$
which lifts the pair $(\pi_1,\pi_0)$ if and only if
\begin{equation}\label{tradelta}
\text{Tra}(\pi_1)^{-1}=\Delta([\pi_0]).
\end{equation}
\end{lemma}

\begin{proof}
(1)
Let $\{\overline{g}\}_{g\in G_0}$ and $\{\overline{\gamma}\}_{\gamma\in \Gamma_0}$ be transversal sets of $G_0$ in $G$
and of $\Gamma_0$ in $\Gamma$ respectively.
These sections determine the 2-place functions
\begin{eqnarray}\label{beta}\begin{array}{rcl}
\beta:G_0\times G_0 & \to & G_1\\
(g_1,g_2)& \mapsto & \overline{g_1}\cdot \overline{g_2}\cdot(\overline{g_1\cdot g_2})^{-1}
\end{array}\end{eqnarray}
and
\begin{eqnarray}\label{omega}\begin{array}{rcl}
\omega:\Gamma_0\times \Gamma_0 & \to & \Gamma_1\\
(\gamma_1,\gamma_2)& \mapsto & \overline{\gamma_1}\cdot \overline{\gamma_2}\cdot(\overline{\gamma_1\cdot \gamma_2})^{-1}.
\end{array}\end{eqnarray}

With this notation, the transgression map is given by
\begin{equation}\label{deftra}
\text{Tra}(\pi_1)=[\pi_1\circ \beta^{-1}]\in H^2(G_0,\Gamma_1),
\end{equation}
where (see \cite[\S 1.1]{K})
$$(\pi_1\circ \beta^{-1})(g_1,g_2):=\pi_1(\beta(g_1,g_2))^{-1}.$$
The coboundary map is given by
\begin{equation}\label{defcob}
\Delta([\pi_0])=[\omega\circ \pi_0]\in H^2(G_0,\Gamma_1),
\end{equation}
where (see \cite[page 124]{s})
$$(\omega\circ \pi_0)(g_1,g_2)=
\overline{\pi_0(g_1)}\cdot {g_1}(\overline{\pi_0(g_2)})\cdot(\overline{\pi_0(g_1\cdot g_2)})^{-1}.$$

Suppose that \eqref{tradelta} holds. Then there exists $\lambda:G_0\to \Gamma_1$
(a 1-coboundary) such that for every $g_1,g_2\in G_0$
\begin{equation}\label{lambda}(\pi_1\circ \beta)(g_1,g_2)\cdot\lambda(g_1\cdot g_2)=
(\omega\circ \pi_0)(g_1,g_2)\cdot\lambda(g_1)\cdot {g_1}(\lambda(g_2)).
\end{equation}
We claim that
$$\begin{array}{rcl}
\pi:G & \to &\Gamma\\
n\cdot\overline{g}& \mapsto &\pi_1(n)\cdot \lambda(g)\cdot \overline{\pi_0(g)},\ \ n\in G_1, g\in G_0
\end{array}$$
is a 1-cocycle (which clearly lifts the pair $(\pi_1,\pi_0)$).
Indeed, for any $n_1\cdot\overline{g_1},n_2\cdot\overline{g_2}\in G$ we have
$$\begin{array}{l}\label{arr}
\pi(n_1\cdot\overline{g_1}\cdot n_2\cdot\overline{g_2})=
\pi(n_1\cdot\overline{g_1}(n_2)\cdot\overline{g_1}\cdot\overline{g_2})=
\pi(n_1\cdot\overline{g_1}(n_2)\cdot\beta(g_1,g_2)\cdot\overline{{g_1}\cdot g_2})=\\
\pi_1(n_1\cdot\overline{g_1}(n_2)\cdot\beta(g_1,g_2))\cdot \lambda({g_1}\cdot g_2)
\cdot\overline{\pi_0({g_1}\cdot g_2})=\text{( by \eqref{lambda})}\\
=\pi_1(n_1)\cdot
(\omega\circ \pi_0)(g_1,g_2)\cdot\lambda(g_1)\cdot {g_1}(\lambda(g_2))\cdot\pi_1(\overline{g_1}(n_2))
\cdot\overline{\pi_0({g_1}\cdot g_2})=\\
\pi_1(n_1)\cdot
\overline{\pi_0(g_1)}\cdot {g_1}(\overline{\pi_0(g_2)})\cdot(\overline{\pi_0(g_1\cdot g_2)})^{-1}
\cdot\lambda(g_1)\cdot {g_1}(\lambda(g_2))\cdot\pi_1(\overline{g_1}(n_2))
\cdot\overline{\pi_0({g_1}\cdot g_2})=\\
\pi_1(n_1)\cdot\lambda(g_1)\cdot\overline{\pi_0(g_1)}
\cdot g_1(\pi_1(n_2)\cdot\lambda(g_2)\cdot\overline{\pi_0(g_2)})
=\pi(n_1\cdot\overline{g_1})\cdot g_1(\pi(n_2\cdot\overline{g_2})).
\end{array}$$

Conversely, suppose that $\pi\in Z^1(G,\Gamma)$ is a 1-cocycle which lifts the pair $(\pi_1,\pi_0)$.
Define
$$\begin{array}{rcl}
\lambda:G_0&\to &\Gamma_1\\
g& \mapsto & \overline{\pi_0(g)}\cdot \pi(\overline{g})^{-1}.
\end{array}$$
Then for every $g_1,g_2\in G_0$

$$\begin{array}{c}(\omega\circ \pi_0)(g_1,g_2)=
\overline{\pi_0(g_1)}\cdot {g_1}(\overline{\pi_0(g_2)})\cdot(\overline{\pi_0(g_1\cdot g_2)})^{-1}=\\
\lambda(g_1)\cdot\pi(\overline{g_1})\cdot g_1(\lambda(g_2)\cdot\pi(\overline{g_2}))
\cdot(\lambda(g_1\cdot g_2)\cdot\pi(\overline{g_1\cdot g_2}))^{-1}=\\
\lambda(g_1)\cdot\pi(\overline{g_1})\cdot g_1(\lambda(g_2)\cdot\pi(\overline{g_2}))
\cdot(\lambda(g_1\cdot g_2)\cdot\pi(\beta(g_1,g_2)^{-1}\cdot\overline{g_1}\cdot \overline{g_2}))^{-1}=\\
\lambda(g_1)\cdot g_1(\lambda(g_2))\cdot\lambda(g_1\cdot g_2)^{-1}\cdot\pi_1(\beta(g_1,g_2)).
\end{array}$$
This proves that $\omega\circ \pi_0$ and $\pi_1\circ\beta$ are cohomologous in $Z^2(G_0,\Gamma_1)$.
Their respective cohomology classes, $\Delta([\pi_0])$ and $\text{Tra}(\pi_1)^{-1}$, are hence equal.
\end{proof}
\begin{remark}\label{uniqueupto}
Under the assumptions of Lemma \ref{lift},
suppose that both $\pi,\pi'\in Z^1(G,\Gamma)$ lift the pair $(\pi_1,\pi_0)$. Define
$$ \begin{array} {rcl}
\pi'':G_0 & \to & \Gamma_1\\
g& \mapsto & \pi'(\overline{g})^{-1}\cdot\pi(\overline{g}).
\end{array}$$
Then the 1-cocycle conditions on $\pi$ and $\pi'$ entail a 1-cocycle condition on $\pi''$.
Moreover, for every $n\in G_1$ and ${g}\in G_0$
$$\pi(n\overline{g})=\pi'(n\overline{g})\cdot\pi''(g).$$
Consequently, a lifting of the pair $(\pi_1,\pi_0)$ is determined up to $\iota^*\circ\inf_G^{G_0}\pi''$
for some $\pi''\in Z^1(G_0,\Gamma_1)$, where $\iota^*$ is the functorial map arising from the embedding
$\Gamma_1\xrightarrow{\iota} \Gamma$ and $\inf_G^{G_0}:Z^1(G_0,\Gamma_1)\to Z^1(G,\Gamma_1)$
is the inflation map.
\end{remark}
\begin{remark}
The criterion \eqref{tradelta} is significant in the theory of lifting
projective representations of $G_0$ to ordinary representations of
$G$ over a field $\mathbb{F}$. Here one puts
$\Gamma:=\text{GL}_n(\mathbb{F})$, and $\Gamma_1:=Z(\text{GL}_n(\mathbb{F}))$ -
the scalar matrices (and so
$\Gamma_0=\Gamma/\Gamma_1=\text{PGL}_n(\mathbb{F})$), endowed with the
trivial $G$-action \cite[Theorem 11.13]{I}.
\end{remark}

\section{Application: Lifting {\it I}-data}\label{exploit}
To exploit Lemma \ref{lift} for our purpose of lifting bijective cocycles, assume both
\begin{enumerate}
\item The extension \eqref{extgamma} is of abelian $G$-modules (via $G_0$), and
\item $G_1=\Gamma_1$ (with the same $G_0$-action), and $\pi_1=$Id$_{G_1}$ is the identity map.
\end{enumerate}
Since by these assumptions the kernel $G_1$ in \eqref{extG} is abelian,
the 2-place function $\beta$ given in \eqref{beta} is a 2-cocycle. That is
\begin{equation}\label{beta1}
[\beta]\in H^2(G_0,G_1)\simeq\text{Ext}^2_{G_0}(\mathbb{Z},G_1).
\end{equation}
The first assumption above says that the extension \eqref{extgamma} determines an element in
$\text{Ext}^1_{G_0}(\Gamma_0,\Gamma_1)$.
By the second assumption, the 2-place function $\omega$ in \eqref{omega} represents a class
$$[\omega]\in\text{Ext}^1_{G_0}(\Gamma_0,G_1).$$
We also have
$$[\pi_0]\in H^1(G_0,\Gamma_0)\simeq\text{Ext}^1_{G_0}(\mathbb{Z},\Gamma_0).$$
Under the above assumptions, the coboundary map $\Delta$ can be identified with the Yoneda splicing \cite[\S 2.6]{b}
of $G_0$-module extensions
\begin{eqnarray}\label{yoneda}
\begin{array}{rcl}
\Delta:\text{Ext}^1_{G_0}(\mathbb{Z},\Gamma_0)&\to &\text{Ext}^2_{G_0}(\mathbb{Z},G_1)\\
{[\pi_0]}
& \mapsto & [\omega] \circ [\pi_0].
\end{array}
\end{eqnarray}
Next, substitution of the identity map Id$_{G_1}$ for $\pi_1$ in \eqref{deftra} yields
\begin{equation}\label{traid}
\text{Tra}(\pi_1)=\text{Tra(Id}_{G_1})=[\text{Id}_{G_1}\circ\beta^{-1}]=[\beta^{-1}].
\end{equation}
We have
\begin{corollary}\label{corlift}
Let \eqref{extG} be a group extension with abelian kernel $G_1$ determined by the class \eqref{beta1}, let
\begin{equation}\label{omegaext}
0\to G_1\to \Gamma\xrightarrow{} \Gamma_0\to 0:[\omega]\in\text{Ext}^1_{G_0}(\Gamma_0,G_1)
\end{equation}
be an exact sequence of abelian
$G$-modules via its quotient $G_0$, and let $\pi_0\in Z^1(G_0,\Gamma_0)$.
Then there exists a 1-cocycle $\pi\in Z^1(G,\Gamma)$
which lifts the pair $(\text{Id}_{G_1},\pi_0)$ if and only if
\begin{equation}\label{tradelta1}
[\beta]= [\omega]\circ [\pi_0]\in H^2(G_0,G_1).
\end{equation}
Moreover, $\pi$ is bijective if and only if so is $\pi_0$.
\end{corollary}
\begin{proof}
The first part is obtained by putting \eqref{yoneda} and \eqref{traid} in Lemma \ref{lift}.
The bijectivity property is verified by a direct computation.
\end{proof}
Note that by Remark \ref{uniqueupto}, the lifting $\pi\in Z^1(G,\Gamma)$ in Corollary \ref{corlift}
is determined up to $\iota^*\circ\inf_G^{G_0}\pi''$
for some $\pi''\in Z^1(G_0,G_1)$, where $\iota^*$ is the functorial map arising from the embedding
$(G_1=)\Gamma_1\xrightarrow{\iota} \Gamma$.
\subsection{}\label{IYBfam}
By now it is clear how Corollary \ref{corlift} is helpful for the
construction of $I$-data on groups
using $I$-data on their quotients. Indeed,
given a bijective 1-cocycle
$\pi_0:G_0\to \Gamma_0$, then for every
extension \eqref{omegaext} of abelian $G_0$-modules, the Yoneda splicing
$[\omega] \circ [\pi_0]\in H^2(G_0,G_1)$ determines a
cover $G$ of $G_0$ and a bijective 1-cocycle $\pi\in
Z^1(G,\Gamma)$ such that $(G,\Gamma,\pi)$ is an $I$-datum ``lying above" the $I$-datum $(G_0,\Gamma_0,\pi_0)$.
The families of IYB groups given in the rest of this subsection demonstrate the technique.
The first example is a special instance of \cite[Theorem 3.4]{CJdR}.
\begin{proposition}\label{sdp}
The family of IYB groups is closed to semidirect products with finite abelian groups.
\end{proposition}
\begin{proof}
The semidirect product $G_1\rtimes G_0$ corresponds to $[\beta]=0\in H^2(G_0,G_1)$ in \eqref{beta1}.
This trivial class is obtained by splicing the cohomology class of the given bijective 1-cocycle
$\pi_0\in Z^1(G_0,\Gamma_0)$ with the trivial $G_0$-extension
$[\omega]=0\in\text{Ext}^1_{G_0}(\Gamma_0,G_1)$. By Corollary
\ref{corlift}, $G_0\rtimes G_1$ admits a bijective 1-cocycle to the direct sum of $G_0$-modules $G_0\oplus \Gamma_0$.
\end{proof}
The following result was given as a consequence of Proposition \ref{sdp} in the published version of this paper. However, it contained an error which was detected and corrected in \cite[\S 2]{CJO18}.
A finite group is said to be {\it of $A$-type} if all its Sylow subgroups are abelian \cite{T}.
\begin{theorem}\label{A-type is IYB}\cite[Theorem 2.1]{CJO18}
Solvable groups of $A$-type are IYB.
\end{theorem}

The following metabelian examples are proven to be IYB by putting $\Gamma_0:=G_0$
as a trivial $G_0$-module in Corollary \ref{corlift}, and letting $\pi_0:=$Id$_{G_0}$ (which is obviously bijective).
Since these families were already treated in \cite{CJdR}, we skip most of the details,
which can be found in \cite[\S 3.3]{bd}.

Let $G_1$ be an abelian $G_0$-module and let $G_1^{G_0}<G_1$ denote its invariant elements under the $G_0$-action.
Classes in the image of the functorial map $$H^2(G_0,G_1^{G_0})\to H^2(G_0,G_1)$$ are termed {\it invariant}. We have
\begin{proposition}\label{invar}\cite[Theorem 3.3.11]{bd}
Let $G_0$ be an abelian group acting trivially on itself and let $G_1$ be an abelian $G_0$-module. Then the map
\begin{eqnarray}\label{class2}
\begin{array}{rcl}
\text{Ext}^1_{G_0}(G_0,G_1)&\to & H^2(G_0,G_1)\\
{[\omega]}
& \mapsto & [\omega] \circ [\text{Id}_{G_0}]
\end{array}
\end{eqnarray} admits all the invariant classes in its image.
\end{proposition}
\begin{corollary}
Let \eqref{extG} be a metabelian extension determined by an invariant class $[\beta]\in H^2(G_0,G_1)$.
Then $G$ is an IYB group. In particular\\
(i) Finite nilpotent groups of class 2 are IYB (see \cite[Corollary 3.11]{CJdR}).\\
(ii) Finite abelian-by-cyclic groups are IYB (see \cite[Corollary 3.10]{CJdR}).
\end{corollary}
\begin{proof}
(i) For a nilpotent group $G$ of class 2, take $G_1$ to be its center. Then the extension \eqref{extG} is
metabelian and central (in particular invariant).\\
(ii) The even dimensional cohomology of a cyclic group $G_0$ with coefficients in an abelian module
is invariant (see e.g. \cite[\S 3.5]{b}).\\
By Corollary \ref{corlift} and Proposition \ref{invar} the outcome groups $G$ in both cases are IYB.
\end{proof}

\subsection{}\label{liftoI}
We can now answer Problem B in cohomological terms.
Suppose that a finite group $G_0$ embeds into $S_n$, that is $\mathbb{Z}^n$ is a faithful
$G_0$-module under the corresponding $n$-permutation action.
Suppose further that $G_0$ admits a module $A$ with $|A|=|G_0|$.
It is not hard to check that rank$(A)<n$.
Then any $G_0$-module surjective map
\begin{equation}\label{surj}
\theta:\mathbb{Z}^n\twoheadrightarrow A
\end{equation}
with finite $G_0$-quotient module $A$ gives rise to a $G_0$-module extension
\begin{equation}\label{extmod}
0\rightarrow \mathbb{Z}^n\rightarrow\mathbb{Z}^n\xrightarrow{\theta} A
\rightarrow 0:[\gamma_{\theta}]\in \Ext^1_{G_0}(A,\mathbb{Z}^n),
\end{equation}
which we call an {\it $n$-fold permutation extension} of $G_0$-modules.
With the notation of \eqref{surj} and \eqref{extmod} we have
\begin{theorem}\label{B}
Let $(G_0,A,\pi_0)$ be an $I$-datum on an IYB group $G_0$.
Then there is a one-to-one correspondence
between groups $G$ of $I$-type, which admit an $n$-fold
$I$-datum $(G,\mathbb{Z}^n,\pi)$, whose associated $I$-datum is $(G_0,A,\pi_0)$,
and $n$-fold permutation extensions $[\gamma_{\theta}]\in \Ext^1_{G_0}(A,\mathbb{Z}^n)$ of $G_0$-modules
(arising from $G_0$-module surjective maps $\theta:\mathbb{Z}^n\twoheadrightarrow A$).
The correspondence is realized by the Yoneda splicing $[\gamma_{\theta}]\mapsto [\gamma_{\theta}]\circ[\pi_0]\in H^2(G_0,\mathbb{Z}^n)$.
\end{theorem}

{\bf Acknowledgments.} We are indebted to E. Aljadeff for his ongoing support.

\end{document}